\documentclass[a4paper,11pt]{article}
\usepackage{amsmath,amsthm,amssymb}
\usepackage[mathscr]{eucal}
\usepackage[bookmarks=false]{hyperref}

{\theoremstyle{definition}\newtheorem{definition}{Definition}

}

\newtheorem{proposition}[definition]{Proposition}
\newtheorem{lemma}[definition]{Lemma}
\newtheorem{theorem}[definition]{Theorem}
\newtheorem{corollary}[definition]{Corollary}

\newcommand{\GL}{\operatorname{GL}}

\newcommand{\M}{\operatorname{M}}
\newcommand{\C}{\mathbb{C}}

\newcommand{\embed}[1]{\prec_{#1}}

\newcommand{\cR}{\mathcal{R}}

\newcommand{\actson}{\curvearrowright}
\newcommand{\SL}{\operatorname{SL}}
\newcommand{\rL}{\operatorname{L}}

\newcommand{\N}{\mathbb{N}}
\newcommand{\T}{\mathbb{T}}
\newcommand{\Z}{\mathbb{Z}}
\newcommand{\cF}{\mathcal{F}}

\newcommand{\cV}{\mathcal{V}}
\newcommand{\id}{\mathord{\operatorname{id}}}

\newcommand{\recht}{\rightarrow}
\newcommand{\cU}{\mathcal{U}}
\newcommand{\vphi}{\varphi}

\newcommand{\eps}{\varepsilon}

\newcommand{\Tr}{\operatorname{Tr}}
\newcommand{\ovt}{\overline{\otimes}}

\newcommand{\ot}{\otimes}

\newcommand{\cG}{\mathcal{G}}

\newcommand{\m}{\mathord{\text{\rm m}}}

\newcommand{\D}{\operatorname{D}}

\newcommand{\PSL}{\operatorname{PSL}}

\newcommand{\rZ}{\operatorname{Z}}

\newcommand{\support}{\operatorname{supp}}
\newcommand{\Ran}{\operatorname{Ran}}
\newcommand{\Comm}{\operatorname{Comm}}
\newcommand{\HNN}{\operatorname{HNN}}

\renewcommand{\Re}{\operatorname{Re}}

\begin{document}

\begin{center}
{\boldmath\LARGE\bf A class of groups for which every action\vspace{0.5ex}\\  is W$^*$-superrigid}

\bigskip

{\sc by Cyril Houdayer\footnote{Partially supported by ANR grant AGORA NT09\textunderscore 461407}, Sorin Popa\footnote{Partially supported by NSF Grant
DMS-0601082} and Stefaan Vaes\footnote{Partially
    supported by ERC Starting Grant VNALG-200749, Research
    Programme G.0639.11 of the Research Foundation --
    Flanders (FWO) and K.U.Leuven BOF research grant OT/08/032.}}
\end{center}

\begin{abstract}\noindent
We prove the uniqueness of the group measure space Cartan subalgebra in crossed products $A \rtimes \Gamma$ covering certain cases where $\Gamma$ is an amalgamated free product over a non-amenable subgroup. In combination with Kida's work we deduce that if $\Sigma < \SL(3,\Z)$ denotes the subgroup of matrices $g$ with $g_{31} = g_{32}=0$, then any free ergodic probability measure preserving action of $\Gamma = \SL(3,\Z)*_\Sigma \SL(3,\Z)$ is stably W*-superrigid. In the second part we settle a technical issue about the unitary conjugacy of group measure space Cartan subalgebras.
\end{abstract}

\section{Introduction}

This short article is a two-fold complement to \cite{PV09}. The main result of \cite{PV09} provides a class $\cG$ of groups $\Gamma$ such that for every free ergodic
probability measure preserving (pmp) action $\Gamma \actson (X,\mu)$, the II$_1$ factor $\rL^\infty(X) \rtimes \Gamma$ has a unique group measure space Cartan subalgebra up to unitary conjugacy. The class $\cG$ contains all non-trivial amalgamated free products $\Gamma = \Gamma_1 *_\Sigma \Gamma_2$ such that $\Gamma$ admits a non-amenable subgroup with the relative property (T) and such that $\Sigma$ is an amenable group that is sufficiently non normal in $\Gamma$. In combination with known orbit equivalence superrigidity theorems, several group actions $\Gamma \actson (X,\mu)$ are shown in \cite{PV09} to be W*-superrigid: any isomorphism between $\rL^\infty(X)\rtimes \Gamma$ and an arbitrary group measure space II$_1$ factor $\rL^\infty(Y)\rtimes \Lambda$, comes from a conjugacy of the actions. For example the Bernoulli action $\Gamma \actson [0,1]^\Gamma$ is W*-superrigid for many of the groups $\Gamma \in \cG$, see \cite[Theorem 1.3]{PV09}. Using Kida's \cite[Theorem 1.4]{
 Ki09} and denoting by $\Sigma < \SL(3,\Z)$ the subgroup of upper triangular matrices, one deduces W*-superrigidity for every free ergodic pmp action $\Gamma \actson (X,\mu)$ such that all finite index subgroups of $\Sigma$ still act ergodically, see \cite[Theorem 6.2]{PV09}.

In the first part of this article we generalize the uniqueness of the group measure space Cartan subalgebra to the case where $\Gamma = \Gamma_1 *_\Sigma \Gamma_2$ is an amalgamated free product over a possibly non-amenable subgroup $\Sigma$, see Theorem \ref{thm.unique-Cartan}. We still assume some softness on $\Sigma$ by imposing the existence of a normal tower $\{e\} = \Sigma_0 \lhd \Sigma_1 \lhd \cdots \lhd \Sigma_{n-1} \lhd \Sigma_n = \Sigma$ such that all quotients $\Sigma_i/\Sigma_{i-1}$ have the Haagerup property. We have to strengthen however the rigidity assumption by imposing that $\Gamma$ admits an infinite subgroup that has property (T). The proof of Theorem \ref{thm.unique-Cartan} is identical to the proof of \cite[Theorem 5.2]{PV09}, apart from the fact that we need a new transfer of rigidity lemma, see Lemma \ref{lem.new-transfer} (cf.\ \cite[Lemma 3.1]{PV09}).

Using Kida's \cite[Theorem 9.11]{Ki09} it follows that if $\Sigma < \SL(3,\Z)$ denotes the subgroup of matrices $g$ with $g_{31} = g_{32}=0$, then any free ergodic pmp action of $\Gamma$ on $(X,\mu)$ is stably W*-superrigid, see Theorem \ref{thm.superrigidity}. Contrary to the case where $\Sigma$ consists of the upper triangular matrices, no ergodicity assumption has to be made on the action of the finite index subgroups of $\Sigma$.

In the second part of this article we provide a detailed argument for the following principle: if $B \subset A \rtimes \Gamma$ is a Cartan subalgebra in a group measure space II$_1$ factor and if $B$ embeds into $A \rtimes \Sigma$ for a sufficiently non normal subgroup $\Sigma < \Gamma$, then $B$ and $A$ are unitarily conjugate. So Proposition \ref{prop.final} provides a justification for the end of the proofs of \cite[Theorems 5.2 and 1.4]{PV09}, which were rather brief compared to the rest of that article. We are very grateful to Steven Deprez who pointed out to us the necessity of adding more details.

\section{Preliminaries}

\subsection*{Intertwining-by-bimodules}

Let $(M,\tau)$ be a von Neumann algebra with a faithful normal tracial state $\tau$ and let $A,B \subset M$ be (possibly non-unital) von Neumann subalgebras. In \cite[Section 2]{Po03} the technique of intertwining-by-bimodules was introduced. It is shown there that the following two conditions are equivalent.
\begin{itemize}
\item There exist projections $p \in A$, $q \in B$, a non-zero partial isometry $v \in p M q$ and a normal $*$-homomorphism $\theta : pAp \recht qBq$ satisfying $x v = v \theta(x)$ for all $x \in pAp$.
\item There is no sequence of unitaries $(w_n)$ in $A$ such that $\|E_B(a w_n b)\|_2 \recht 0$ for all $a, b \in M$.
\end{itemize}
If one, and hence both, of these conditions hold, we write $A \prec_M B$. By \cite[Theorem A.1]{Po01}, if $A$ and $B$ are Cartan subalgebras of a II$_1$ factor $M$, then $A \prec_M B$ if and only if $A$ and $B$ are unitarily conjugate.

\subsection*{Property (T) for von Neumann algebras}

Let $(P,\tau)$ be a von Neumann algebra with a faithful normal tracial state $\tau$. A normal completely positive map $\vphi : P \recht P$ is said to be subunital if $\vphi(1) \leq 1$ and subtracial if $\tau \circ \vphi \leq \tau$. Following \cite{CJ85}, we say that $P$ has property (T) if every sequence of normal subunital subtracial completely positive maps $\vphi_n : P \recht P$ converging to the identity pointwise in $\|\,\cdot\,\|_2$, converges uniformly in $\|\,\cdot\,\|_2$ on the unit ball of $P$.

If $\Gamma$ is a countable group, then $\rL \Gamma$ has property (T) if and only if the group $\Gamma$ has property (T) in the usual sense.

\subsection*{Relative property (H) and property anti-(T)}

In \cite[Section 2]{Po01} property (H) of a finite von Neumann algebra $M$ relative to a von Neumann subalgebra $P \subset M$ is introduced. We recall from \cite{Po01} the following facts.
\begin{itemize}
\item If $\Gamma \actson P$ is a trace preserving action, then $P \rtimes \Gamma$ has property (H) relative to $P$ if and only if the group $\Gamma$ has the Haagerup property. Recall that a countable group $\Gamma$ has the Haagerup property if and only if there exists a sequence of positive definite functions $\vphi_n : \Gamma \recht \C$ tending to $1$ pointwise and such that for every $n$ the function $\vphi_n$ belongs to $c_0(\Gamma)$ (see e.g.\ \cite{CCJJV}).
\item If $M$ has property (H) relative to $P \subset M$, there exists a sequence of normal subunital subtracial completely positive $P$-bimodular maps $\vphi_n : M \recht M$ such that $\|\vphi_n(x) - x \|_2 \recht 0$ for every $x \in M$ and such that every $\vphi_n$ satisfies the following relative compactness property: if $(w_k)$ is a sequence of unitaries in $M$ satisfying $\|E_P(aw_k b)\|_2 \recht 0$ for all $a,b \in M$, then $\|\vphi_n(w_k)\|_2 \recht 0$ when $k \recht \infty$. The converse is almost true, but we have no need to go into these technical details.
\item If $M$ has property (H) relative to $P \subset M$, then $N \ovt M$ has property (H) relative to $N \ovt P$ for every finite von Neumann algebra $N$.
\end{itemize}

The following lemma is essentially contained in \cite[Theorem 6.2]{Po01}. We provide a full proof for the convenience of the reader.

\begin{lemma} \label{lem.technical}
Let $(M,\tau)$ be a tracial von Neumann algebra and $P_1 \subset P \subset M$ von Neumann subalgebras. Assume that $P$ has property (H) relative to $P_1$.

If $M_0 \subset pMp$ is a von Neumann subalgebra such that $M_0$ has property (T) and $M_0 \prec_M P$, then $M_0 \prec_M P_1$.
\end{lemma}
\begin{proof}
Assume that $M_0 \not\prec_M P_1$. Since $M_0 \prec_M P$, by \cite[Remark 3.8]{Va07} we find non-zero projections $p_0 \in M_0$, $q \in P$, a non-zero partial isometry $v \in p_0 M q$ and a normal unital $*$-homomorphism $\theta : p_0 M_0 p_0 \recht q P q$ such that $x v = v \theta(x)$ for all $x \in p_0 M_0 p_0$ and such that $\theta(p_0 M_0 p_0) \not\prec_P P_1$.

Since $P$ has property (H) relative to $P_1$, we can take a sequence $\vphi_n : P \recht P$ of normal subunital subtracial completely positive maps such that $\|\vphi_n(x) - x \|_2 \recht 0$ for all $x \in P$ and such that every $\vphi_n$ satisfies the relative compactness property explained above. Since $\theta(p_0 M_0 p_0)$ has property (T), take $n$ such that
$\|\vphi_n(w) - w \|_2 \leq \|q\|_2 / 2$ for all unitaries $w \in \theta(p_0 M_0 p_0)$. Since $\theta(p_0 M_0 p_0) \not\prec_P P_1$, take a sequence of unitaries $w_k \in \theta(p_0 M_0 p_0)$ such that $\|E_{P_1}(a w_k b)\|_2 \recht 0$ for all $a,b \in P$. By the relative compactness of $\vphi_n$, it follows that $\|\vphi_n(w_k)\|_2 \recht 0$ when $k \recht \infty$. So, for $k$ large enough, we have $\|\vphi_n(w_k)\|_2 < \|q\|_2 / 2$. It follows that
$$\|q\|_2 = \|w_k\|_2 \leq \|\vphi_n(w_k)\|_2 + \|\vphi_n(w_k) - w_k \|_2 < \|q\|_2 \; ,$$
which is absurd.
\end{proof}

We say that a finite von Neumann algebra $P$  is \emph{anti-(T)} if there exist von Neumann subalgebras $\C1 = P_0 \subset P_1 \subset \cdots \subset P_n = P$ such that for all $i=1,\ldots,n$, the von Neumann algebra $P_i$ has property (H) relative to $P_{i-1}$. Repeatedly applying Lemma \ref{lem.technical}, an anti-(T) von Neumann algebra cannot contain a diffuse von Neumann subalgebra with property (T).

We say that a countable group $\Sigma$ is anti-(T) if there exist subgroups $\{e\} = \Sigma_0 < \Sigma_1 < \cdots < \Sigma_n = \Sigma$ such that for all $i=1,\ldots,n$, $\Sigma_{i-1}$ is normal in $\Sigma_i$ and $\Sigma_i/\Sigma_{i-1}$ has the Haagerup property. If $\Sigma$ is anti-(T), then the group von Neumann algebra $\rL \Sigma$ is anti-(T) as well. An anti-(T) group cannot contain an infinite subgroup with property (T). Nevertheless, $\SL(2,\Z) \ltimes \Z^2$ is an anti-(T) group (since $\SL(2,\Z)$ has the Haagerup property and $\Z^2$ is amenable) which contains an infinite subgroup with the \emph{relative} property (T), namely $\Z^2$. This explains why our new transfer of rigidity lemma (see Lemma \ref{lem.new-transfer}) requires property (T) rather than relative property (T).

\section{Transfer of rigidity and W*-superrigidity}

We say that free ergodic pmp actions $\Gamma \actson (X,\mu)$ and $\Lambda \actson (Y,\eta)$ are
\begin{itemize}
\item W*-equivalent, if $\rL^\infty(X) \rtimes \Gamma \cong \rL^\infty(Y) \rtimes \Lambda$,
\item orbit equivalent, if the orbit equivalence relations $\cR(\Gamma \actson X)$ and $\cR(\Lambda \actson Y)$ are isomorphic,
\item conjugate, if there exists an isomorphism of probability spaces $\Delta : X \recht Y$ and an isomorphism of groups $\delta : \Gamma \recht \Lambda$ such that $\Delta(g \cdot x) = \delta(g) \cdot \Delta(x)$ almost everywhere.
\end{itemize}

Following \cite[Definition 6.1]{PV09} a free ergodic pmp action $\Gamma \actson (X,\mu)$ is called \emph{W*-superrigid} if the following property holds. If $\Lambda \actson (Y,\eta)$ is an arbitrary free ergodic pmp action and $\pi : \rL^\infty(X) \rtimes \Gamma \recht \rL^\infty(Y) \rtimes \Lambda$ is a W*-equivalence, then the actions $\Gamma \actson X$ and $\Lambda \actson Y$ are conjugate through $\Delta : X \recht Y$, $\delta : \Gamma \recht \Lambda$ and up to unitary conjugacy $\pi$ is of the form
$$\pi(a u_g) = \Delta_*(a \omega_g) u_{\delta(g)} \quad\text{for all}\quad a \in \rL^\infty(X), g \in \Gamma,$$
where $(\omega_g) \in \rZ^1(\Gamma \actson X)$ is a $\T$-valued $1$-cocycle for the action $\Gamma \actson X$.

Slightly more natural than W*-superrigidity is the notion of \emph{stable W*-superrigidity} where possible finite index issues are correctly taken into account. A stable isomorphism between II$_1$ factors $M$ and $N$ is an isomorphism between $M$ and an amplification $N^t$. This leads to the notion of \emph{stable W*-equivalence} between free ergodic pmp actions. Similarly one defines \emph{stable orbit equivalence.} Finally, a \emph{stable conjugacy} between two free ergodic pmp actions $\Gamma \actson (X,\mu)$ and $\Lambda \actson (Y,\eta)$ is a conjugacy between the actions $\frac{\Gamma_0}{G} \actson \frac{X_0}{G}$ and $\frac{\Lambda_0}{H} \actson \frac{Y_0}{H}$ where $\Gamma \actson X$, $\Lambda \actson Y$ are induced\footnote{A free ergodic pmp action $\Gamma \actson (X,\mu)$ is said to be induced from $\Gamma_0 \actson X_0$ if $\Gamma_0 < \Gamma$ is a finite index subgroup and $X_0 \subset X$ is a non-negligible $\Gamma_0$-invariant subset such that $\mu(X_0 \cap g \cdot X_0) = 0$ for all $g \in \Gamma - \Gamma_0$.} from $\Gamma_0 \actson X_0$, $\Lambda_0 \actson Y_0$ and where $G \lhd \Gamma_0$, $H \lhd \Lambda_0$ are finite normal subgroups.

\begin{definition}[{\cite[Definition 6.4]{PV09}}] \label{def.stable-W-superrigid}
A free ergodic pmp action $\Gamma \actson (X,\mu)$ is said to be \emph{stably W$^*$-superrigid} if the following holds. Whenever $\pi$ is a stable W$^*$-equivalence between $\Gamma \actson (X,\mu)$ and an arbitrary free ergodic pmp action $\Lambda \actson (Y,\eta)$, the actions are stably conjugate and $\pi$ equals the composition of
\begin{itemize}
\item the canonical stable W$^*$-equivalence given by the stable conjugacy,
\item the automorphism of $\rL^\infty(X) \rtimes \Gamma$ given by an element of $\rZ^1(\Gamma \actson X)$,
\item an inner automorphism of $\rL^\infty(X) \rtimes \Gamma$.
\end{itemize}
\end{definition}

Let $\Gamma \actson (X,\mu)$ be stably W$^*$-superrigid. If moreover $\Gamma$ has no finite normal subgroups and if finite index subgroups of $\Gamma$ still act ergodically on $(X,\mu)$, then $\Gamma \actson (X,\mu)$ is W$^*$-superrigid in the sense explained above.

The following is the main result in this section.

\begin{theorem}\label{thm.superrigidity}
Denote by $\Sigma < \SL(3,\Z)$ the subgroup of matrices $g$ such that $g_{31} = g_{32} = 0$. Put $\Gamma = \SL(3,\Z) *_\Sigma \SL(3,\Z)$. Every free ergodic pmp action $\Gamma \actson (X,\mu)$ is stably W*-superrigid. In particular all a-periodic\footnote{A free ergodic pmp action is called a-periodic if it is not induced from a finite index subgroup, i.e.\ if finite index subgroups still act ergodically.} free ergodic pmp actions of $\Gamma$ are W*-superrigid.

More generally, if $k \geq 1$ and $n_1,\ldots,n_k \in \{1,2\}$, the same conclusion holds for the group $\Gamma = \PSL(n,\Z) *_\Sigma \PSL(n,\Z)$ where $n = 2 + n_1 + \cdots + n_k$ and $\Sigma$ is the image in $\PSL(n,\Z)$ of
$$\SL(n,\Z) \cap \begin{pmatrix} \GL(2,\Z) & * & \cdots & * \\ 0 & \GL(n_1,\Z) & \cdots & * \\ \vdots  & \vdots & \ddots & \vdots
\\ 0 & 0 & \cdots & \GL(n_k,\Z) \end{pmatrix}\; .$$
\end{theorem}
\begin{proof}
The theorem is a direct consequence of Kida's \cite[Theorem 9.11]{Ki09} and the uniqueness of group measure space Cartan theorem \ref{thm.unique-Cartan} below. Let
$$P =  \begin{pmatrix} 1 & * & \cdots & * \\ 0 & 1 & \cdots & * \\ \vdots  & \vdots & \ddots & \vdots
\\ 0 & 0 & \cdots & 1 \end{pmatrix} \mbox{ and } G = \SL(n,\Z) \cap \begin{pmatrix} \GL(2,\Z) & 0 & \cdots & 0 \\ 0 & \GL(n_1,\Z) & \cdots & 0 \\ \vdots  & \vdots & \ddots & \vdots
\\ 0 & 0 & \cdots & \GL(n_k,\Z) \end{pmatrix}\; ,$$
and denote by $\overline P$ (resp.\ $\overline G$) the image of $P$ (resp.\ $G$) in $\PSL(n, \Z)$. We have that $\overline P$ is amenable and normal in $\Sigma$ and $\overline G \cong \Sigma/ \overline P$ has the Haagerup property. This shows that $\Sigma$ is anti-(T).
Therefore, if $\Gamma \actson (X,\mu)$ is an arbitrary free ergodic pmp action, Theorem \ref{thm.unique-Cartan} says that every stable W*-equivalence between $\Gamma \actson (X,\mu)$ and an arbitrary $\Lambda \actson (Y,\eta)$ comes from a stable orbit equivalence of the actions. Kida showed in \cite[Theorem 9.11]{Ki09} that $\Gamma$ is coupling rigid with respect to the abstract commensurator\footnote{Given a group $\Gamma$ the abstract commensurator $\Comm(\Gamma)$ is defined as the group of all isomorphisms $\delta : \Gamma_1 \recht \Gamma_2$ between finite index subgroups $\Gamma_1,\Gamma_2 < \Gamma$, identifying two such isomorphisms when they coincide on a finite index subgroup. Inner conjugacy provides a homomorphism from $\Gamma$ to $\Comm(\Gamma)$, which is injective if and only if $\Gamma$ is icc.} $\Comm(\Gamma)$. Since $\Gamma$ is icc and $\Comm(\Gamma)$ is countable, this precisely means that every stable orbit equivalence comes from a stable conjugacy, cf.\ \cite[Proposition 3.11]{Ki09}.
\end{proof}

The W*-superrigidity in Theorem \ref{thm.superrigidity} arises as the combination of Kida's OE superrigidity and the following uniqueness result for group measure space Cartan subalgebras. We first need a new transfer of rigidity lemma (cf.\ \cite[Lemma 3.1]{PV09}).

\begin{lemma}\label{lem.new-transfer}
Let $M$ be a {\rm II}$_1$ factor and $\vphi_n : M \recht M$ a sequence of normal subunital subtracial completely positive maps. Assume that $P,M_0 \subset M$ are von Neumann subalgebras such that $P$ is anti-(T) and such that $M_0$ is diffuse and has property (T).

Let $p \in M$ be a projection and $pMp = Q \rtimes \Lambda$ any crossed product decomposition with $Q$ being anti-(T). Denote by $(v_s)_{s \in \Lambda}$ the corresponding canonical unitaries.

For every $\eps > 0$, there exists $n$ and a sequence $(s_k)_{k \in \N}$ in $\Lambda$ such that
\begin{enumerate}
\item $\|\vphi_n(v_{s_k}) - v_{s_k}\|_2 \leq \eps$ for all $k \in \N$,
\item for all $a,b \in M$ we have $\|E_P(a v_{s_k} b)\|_2 \recht 0$ when $k \recht \infty$.
\end{enumerate}
\end{lemma}

\begin{proof}
Since $M$ is a II$_1$ factor and $M_0$ is diffuse, we may assume that $p \in M_0$. Write $N = Q \rtimes \Lambda$ and denote by $\Delta : N \recht N \ovt N$ the normal $*$-homomorphism given by $\Delta(a v_s) = a v_s \ot v_s$ for all $a \in Q$ and $s \in \Lambda$. Normalize the trace $\tau$ on $M$ such that $\tau(p) = 1$.

Since $Q$ is anti-(T), Lemma \ref{lem.technical} implies that $pM_0 p \not\prec_N Q$. Let $(w_n)$ be a sequence of unitaries in $pM_0 p$ such that $\|E_Q(a w_n b)\|_2 \recht 0$ for all $a,b \in N$. It follows that $\Delta(w_n)$ is a sequence of unitaries in $N \ovt N$ satisfying
$$\|(\id \ot \tau)(a \Delta(w_n) b)\|_2 \recht 0 \quad\text{for all}\quad a,b \in N \ovt N \; .$$
Indeed, it suffices to check the convergence for $a = 1 \otimes c v_s$ and $b = 1 \otimes d v_t$, where $c, d \in Q$ and $s, t \in \Lambda$. We have $\|(\id \ot \tau)((1 \otimes c v_s) \Delta(w_n) (1 \otimes d v_t))\|_2 = |\tau(\sigma_{t^{-1}}(d) c)| \|E_Q(w_n v_{ts})\|_2 \recht 0$. So, $\Delta(pM_0 p) \not\prec_{N \ovt N} N \ot 1$. Since $P$ is anti-(T), Lemma \ref{lem.technical} implies that $\Delta(pM_0 p) \not\prec_{N \ovt M} N \ovt P$.

Choose $\eps > 0$. Put $\eps_1 = \eps^2/4$. Since $pM_0 p$ has property (T), take $n$ such that
$$1-\Re (\tau \ot \tau)(\Delta(w)^* \, (\id \ot \vphi_n)\Delta(w)) \leq \eps_1$$
for all $w \in \cU(pM_0 p)$. Define
$$\cV := \{s \in \Lambda \mid 1-\Re \tau (v_s^* \vphi_n(v_s)) \leq 2 \eps_1 \} \; .$$
Note that for all $s \in \cV$, we have $\|\vphi_n(v_s) - v_s\|_2 \leq \sqrt{4 \eps_1} = \eps$. In order to prove the lemma, it suffices to show that there exists a sequence $s_k \in \cV$ such that $\|E_P(a v_{s_k} b)\|_2 \recht 0$ for all $a, b \in M$. Assume the contrary. We then find a finite subset $\cF \subset M$ and a $\delta > 0$ such that
$$\sum_{a,b \in \cF} \|E_P(a v_s b)\|_2^2 \geq \delta \quad\text{for all}\quad s \in \cV \; .$$
We will deduce that $\Delta(pM_0 p) \prec_{N \ovt M} N \ovt P$. This will be a contradiction with the statement $\Delta(pM_0 p) \not\prec_{N \ovt M} N \ovt P$ that we have shown at the beginning of the proof.

Let $w \in \cU(pM_0 p)$. Write $w = \sum_{s \in \Lambda} w_s v_s$ where $w_s \in Q$. Since $\tau(p) = 1$, we have $\sum_{s \in \Lambda} \|w_s\|_2^2 = 1$. Therefore,
\begin{align*}
\eps_1 & \geq 1-\Re (\tau \ot \tau)(\Delta(w)^* \, (\id \ot \vphi_n)\Delta(w)) \\
& = \sum_{s \in \Lambda} \|w_s\|_2^2 \; \bigl(1- \Re \tau (v_s^* \vphi_n(v_s)) \bigr) \\
& \geq \sum_{s \in \Lambda-\cV} \|w_s\|_2^2 \; \bigl(1- \Re \tau (v_s^* \vphi_n(v_s)) \bigr) \\
& \geq \sum_{s \in \Lambda-\cV} \|w_s\|_2^2 \; 2 \eps_1 \; .
\end{align*}
We conclude that for all $w \in \cU(pM_0 p)$ we have
$$\sum_{s \in \Lambda-\cV} \|w_s\|_2^2 \leq \frac{1}{2}$$
implying that
$$\sum_{s \in \cV} \|w_s\|_2^2 \geq \frac{1}{2}$$
for all $w \in \cU(pM_0 p)$.

It follows that for all $w \in \cU(pM_0 p)$
\begin{align*}
\sum_{a,b \in \cF}  \|E_{N \ovt P}( (1 \ot a) \Delta(w) (1 \ot b) )\|_2^2 & = \sum_{a,b \in \cF, s \in \Lambda} \|w_s\|_2^2 \; \|E_P(a v_s b)\|_2^2 \\
& \geq \sum_{a,b \in \cF, s \in \cV} \|w_s\|_2^2 \; \|E_P(a v_s b)\|_2^2 \\
& \geq \sum_{s \in \cV} \|w_s\|_2^2 \; \delta \\
& \geq \frac{\delta}{2} \; .
\end{align*}
This means that $\Delta(pM_0 p) \prec_{N \ovt M} N \ovt P$. We have reached the desired contradiction.
\end{proof}

We are ready to formulate and prove our uniqueness of group measure space Cartan theorem. We use the notation $\D_n(\C) \subset \M_n(\C)$ to denote the subalgebra of diagonal matrices.

\begin{theorem}\label{thm.unique-Cartan}
Let $\Gamma = \Gamma_1 *_\Sigma \Gamma_2$ be an amalgamated free product satisfying the following conditions: $\Gamma_1$ admits an infinite subgroup with property (T), $\Sigma$ is anti-(T) and $\Gamma_2 \neq \Sigma$. Assume moreover that there exist $g_1,\ldots,g_n \in \Gamma$ such that $\bigcap_{i=1}^n g_i \Sigma g_i^{-1}$ is finite. Let $\Gamma \actson (X,\mu)$ be any free ergodic pmp action and denote $M = \rL^\infty(X) \rtimes \Gamma$.

Whenever $\Lambda \actson (Y,\eta)$ is a free ergodic pmp action, $p \in \M_n(\C) \ot M$ is a projection and
$$\pi : \rL^\infty(Y) \rtimes \Lambda \recht p (\M_n(\C) \ot M)p$$
is a $*$-isomorphism, there exists a projection $q \in \D_n(\C) \ot \rL^\infty(X)$ and a unitary $u \in q(\M_n(\C) \ot M)p$ such that
$$\pi(\rL^\infty(Y)) = u^* (\D_n(\C) \ot \rL^\infty(X)) u \; .$$
\end{theorem}
\begin{proof}
Write $A = \M_n(\C) \ot \rL^\infty(X)$ and $N = A \rtimes \Gamma$. Put $B = \rL^\infty(Y)$. We first prove that $\pi(B) \prec_N A \rtimes \Sigma$. Denote by $|g|$ the length of $g \in \Gamma_1 *_\Sigma \Gamma_2$ as a reduced word. Denote for $0 < \rho < 1$ by $\m_\rho$ the corresponding completely positive maps on $N$ given by $\m_\rho(a u_g) = \rho^{|g|} a u_g$ for all $a \in A$, $g \in \Gamma$. When $\rho \recht 1$, we have $\m_\rho \recht \id$ pointwise in $\|\, \cdot \,\|_2$.

Write $\eps = \tau(p)/5000$ and $P = A \rtimes \Sigma$. Note that $P$ is anti-(T). By Lemma \ref{lem.new-transfer} take $0 < \rho_1 < 1$ and a sequence $(s_k)$ in $\Lambda$ satisfying $\|m_{\rho_1}(v_{s_k}) - v_{s_k}\|_2 \leq \eps$ for all $k$ and $\|E_P(x v_{s_k} y)\|_2 \recht 0$ for all $x,y \in N$. By \cite[Lemma 5.7]{PV09} there exists a $0 < \rho < 1$ and a $\delta > 0$ such that $\tau(w^* \m_\rho(w)) \geq \delta$ for all $w \in \cU(\pi(B))$. By \cite[Theorem 5.4]{PV09} and because $\pi(B)$ is regular in $pNp$, it follows that $\pi(B) \prec_N P$.

So we have shown that $\pi(B) \prec_N A \rtimes \Sigma$. By Proposition \ref{prop.final} below and since we have $g_1,\ldots,g_n \in \Gamma$ such that $\bigcap_{i=1}^n g_i \Sigma g_i^{-1}$ is finite, it follows that $\pi(B) \prec_N A$. The theorem now follows from \cite[Theorem A.1]{Po01}.
\end{proof}

\section{An embedding result, strengthening \cite[Theorem 6.16]{PV06}}

Assume that $(A,\tau)$ is a tracial von Neumann algebra and that $\Gamma \overset{\sigma}{\actson} (A,\tau)$ is a trace preserving action. We do not assume that $\sigma$ is properly outer or that $\sigma$ is ergodic. Let $M = A \rtimes \Gamma$.

Whenever $\Lambda < \Gamma$ is a subgroup, consider the basic construction $\langle M, e_{A \rtimes \Lambda} \rangle$. By definition $\langle M,e_{A \rtimes \Lambda} \rangle$ consists of those operators on $\rL^2(M)$ that commute with the right module action of $A \rtimes \Lambda$. The basic construction comes with a canonical operator valued weight $T_\Lambda$ from $\langle M, e_{A \rtimes \Lambda} \rangle^+$ to the extended positive part of $M$. For all $x,y \in M$, the element $x e_{A \rtimes \Lambda} y$ is integrable with respect to $T_\Lambda$ and $T_\Lambda(x e_{A \rtimes \Lambda} y) = xy$. Choose $g_i \in \Gamma$ such that $\Gamma = \bigsqcup_{i} g_i  \Lambda = \bigsqcup_{i} \Lambda g_i^{-1}$.
Denoting by $\rho_g$ the right multiplication operator by $u_g^*$ on $\rL^2(M)$, one checks that $\sum_i \rho_{g_i} e _{A \rtimes \Lambda} \rho_{g_i}^* = 1$, whence
$$T_\Lambda(x) = \sum_{i} \rho_{g_i} x \rho_{g_i}^* \quad\text{for all}\quad x \in \langle M, e_{A \rtimes \Lambda} \rangle^+ \; .$$
The canonical semi-finite trace $\Tr_\Lambda$ on $\langle M, e_{A \rtimes \Lambda} \rangle$ is given as the composition of $T_\Lambda$ and the trace $\tau$ on $M$.

Assume that $p \in M$ is a projection and $B \subset pMp$ is a quasi-regular subalgebra (see \cite[1.4.2]{Po01}). Recall that this means that the quasi-normalizer of $B$ inside $pMp$ is weakly dense in $pMp$. Obviously, regular subalgebras $B \subset pMp$ or, even more specifically, Cartan subalgebras $B \subset pMp$ are quasi-regular.

Given a subgroup $\Lambda < \Gamma$, let $H \subset p\rL^2(M)$ be the closed linear span of all $B$-$(A \rtimes \Lambda)$-subbimodules of $p\rL^2(M)$ that are finitely generated as a right Hilbert $(A \rtimes \Lambda)$-module. Since $B \subset pMp$ is quasi-regular, $H$ is stable by left multiplication with $pMp$. The subspace $H$ is also invariant under right multiplication by $A \rtimes \Lambda$. So $H$ is of the form $p\rL^2(M)z(\Lambda)$ for some projection $z(\Lambda) \in M \cap (A \rtimes \Lambda)'$. We make $z(\Lambda)$ uniquely defined by requiring that $z(\Lambda)$ is smaller or equal than the central support of $p$ in $M$.

Note that by definition $z(\Lambda) \neq 0$ if and only if $B \embed{M} A \rtimes \Lambda$.

Denote by $J : \rL^2(M) \recht \rL^2(M)$ the adjoint operator. So, given $x \in M$, $Jx^*J$ is the operator of right multiplication with $x$. Denote by $\support a$ the support projection of a positive operator $a$. Observe that
\begin{equation}\label{eq.benader}
\begin{split}
p & J z(\Lambda)J  \\ &= \bigvee \{q \mid q\;\;\text{is an orthogonal projection in $B' \cap p\langle M,e_{A \rtimes \Lambda} \rangle p$ satisfying}\;\; \Tr_\Lambda(q) < \infty\} \\
& = \bigvee \{q \mid q\;\;\text{is an orthogonal projection in $B' \cap p \langle M,e_{A \rtimes \Lambda} \rangle p$ satisfying}\;\; \|T_\Lambda(q)\| < \infty\} \\
& = \bigvee \{\support a \mid a \in (B' \cap p \langle M,e_{A \rtimes \Lambda} \rangle p)^+ \;\;\text{and}\;\; \|T_\Lambda(a)\| < \infty \} \; .
\end{split}
\end{equation}
If $a$ and $b$ are positive operators, then $\support (a) \vee \support(b) = \support(a+b)$. So we find a sequence of elements $a_n \in (B' \cap p \langle M,e_{A \rtimes \Lambda} \rangle p)^+$ such that all $T_\Lambda(a_n)$ are bounded and $\support(a_n) \recht p J z(\Lambda)J$ strongly. Moreover, every projection $\support(a_n)$ can be strongly approximated by a spectral projection of the form $q_n = \chi_{[\eps_n,+\infty)}(a_n)$ for $\eps_n > 0$ sufficiently small. We have $q_n \leq \frac{1}{\eps_n} a_n$ so that $T_\Lambda(q_n)$ is bounded. Hence we find a sequence of projections $q_n \in B' \cap p \langle M,e_{A \rtimes \Lambda} \rangle p$ such that $q_n \recht p Jz(\Lambda)J$ strongly and $\|T_\Lambda(q_n)\| < \infty$ for all $n$.

Note that $z(\Lambda_1) \leq z(\Lambda_2)$ when $\Lambda_1 < \Lambda_2 < \Gamma$. Indeed, it suffices to observe that for every $B$-$(A \rtimes \Lambda_1)$-subbimodule $K \subset p \rL^2(M)$ that is finitely generated as a right Hilbert module, the closed linear span of $K (A \rtimes \Lambda_2)$ is a $B$-$(A \rtimes \Lambda_2)$-subbimodule of $p \rL^2(M)$ that is finitely generated as a right Hilbert module. Hence $p \rL^2(M) z(\Lambda_1) \subset p \rL^2(M) z(\Lambda_2)$. Since by convention $z(\Lambda_1)$ and $z(\Lambda_2)$ are smaller or equal than the central support of $p$, we conclude that $z(\Lambda_1) \leq z(\Lambda_2)$.

By definition, $z(g\Lambda g^{-1}) = u_g z(\Lambda) u_g^*$ for all $g \in \Gamma$, $\Lambda < \Gamma$.

\begin{proposition}\label{prop.intersection}
Let $M = A \rtimes \Gamma$ for some trace preserving action of a countable group $\Gamma$ on the tracial von Neumann algebra $(A,\tau)$. Let $B \subset pMp$ be a quasi-regular von Neumann subalgebra. For every subgroup $\Lambda < \Gamma$, define as above the projection $z(\Lambda) \in M \cap (A \rtimes \Lambda)'$ such that $p \rL^2(M) z(\Lambda)$ equals the closed linear span of all $B$-$(A \rtimes \Lambda)$-subbimodules of $p \rL^2(M)$ that are finitely generated as a right Hilbert module.

If $\Sigma < \Gamma$ and $\Lambda < \Gamma$ are subgroups, then the projections $z(\Sigma)$ and $z(\Lambda)$ commute and satisfy
$$z(\Sigma \cap \Lambda) = z(\Sigma) \, z(\Lambda) \; .$$
\end{proposition}
\begin{proof}
We use the operator valued weight $T_\Sigma$ as explained above. As we saw after formulae \eqref{eq.benader} we can take projections $q_n \in B' \cap p \langle M,e_{A \rtimes \Sigma}\rangle p$ and $e_n \in B' \cap p \langle M,e_{A \rtimes \Lambda}\rangle p$ such that $q_n \recht p Jz(\Sigma)J$ and $e_n \recht p Jz(\Lambda)J$ strongly and such that for all $n \in \N$, we have $\|T_\Sigma(q_n)\| < \infty$ and $\|T_\Lambda(e_n)\| < \infty$. Note that $e_n q_n e_n \in (B' \cap p \langle M,e_{A \rtimes (\Sigma \cap \Lambda)}\rangle p)^+$. We claim that
\begin{equation}\label{eq.claim}
\|T_{\Sigma \cap \Lambda}(e_n q_n e_n)\| < \infty \quad\text{for all}\;\; n \in \N \; .
\end{equation}
Fix $n \in \N$. Take $g_i \in \Gamma$ such that $\Gamma = \bigsqcup_{i} g_i \Lambda$. Take $h_j \in \Lambda$ such that $\Lambda = \bigsqcup_j h_j (\Sigma \cap \Lambda)$. Note that the cosets $h_j \Sigma$ are disjoint and that $\Gamma = \bigsqcup_{i,j} g_i h_j (\Sigma \cap \Lambda)$. Because of the latter, we have
$$T_{\Sigma \cap \Lambda}(e_n q_n e_n) = \sum_{i,j} \rho_{g_i h_j} e_n q_n e_n \rho_{g_i h_j}^* = \sum_i \rho_{g_i} e_n \Bigl(\sum_j \rho_{h_j} q_n \rho_{h_j}^*\Bigr) e_n \rho_{g_i}^* \; .$$
Because the cosets $h_j \Sigma$ are disjoint, we know that
$$\sum_j \rho_{h_j} q_n \rho_{h_j}^* \leq \sum_{h \in \Gamma/\Sigma} \rho_h q_n \rho_h^* = T_\Sigma(q_n) \leq \|T_\Sigma(q_n)\| \, 1 \; .$$
Therefore,
$$T_{\Sigma \cap \Lambda}(e_n q_n e_n) \leq \|T_\Sigma(q_n)\| \; \sum_i \rho_{g_i} e_n \rho_{g_i}^* = \|T_\Sigma(q_n)\| \, T_\Lambda(e_n) \; .$$
So claim \eqref{eq.claim} is proven.

Since $e_n q_n e_n \in (B' \cap p \langle M,e_{A \rtimes (\Sigma \cap \Lambda)}\rangle p)^+$ claim \eqref{eq.claim} implies that
$$\support(e_n q_n e_n) \leq p J z(\Sigma \cap \Lambda) J \quad\text{for all}\;\; n \in \N \; .$$
Hence, $\Ran (e_n q_n) \subset p \rL^2(M) z(\Sigma \cap \Lambda)$ for all $n$. Since $e_n \recht p Jz(\Lambda) J$ and $q_n \recht p J z(\Sigma) J$ strongly and since $z(\Lambda)$ and $z(\Sigma)$ are chosen below the central support of $p$, it follows that
$$\Ran (z(\Lambda) z(\Sigma)) \subset \Ran z(\Sigma \cap \Lambda) \; .$$
Since $z(\Sigma \cap \Lambda) \leq z(\Sigma)$ and $z(\Sigma \cap \Lambda) \leq z(\Lambda)$, we get the chain of inclusions
$$\Ran z(\Lambda) \cap \Ran z(\Sigma) \subset \Ran (z(\Lambda) z(\Sigma)) \subset \Ran z(\Sigma \cap \Lambda) \subset \Ran z(\Lambda) \cap \Ran z(\Sigma) \; .$$
So all these inclusions are equalities. This means that $z(\Lambda)$ and $z(\Sigma)$ are commuting projections and
$z(\Sigma \cap \Lambda) = z(\Sigma) \, z(\Lambda)$.
\end{proof}

As an immediate corollary we find the following generalization of \cite[Theorem 6.16]{PV06}.

\begin{corollary}
Let $\Gamma \actson A$ be a trace preserving action, $M = A \rtimes \Gamma$, $\Sigma < \Gamma$ a subgroup and $B \subset pMp$ a quasi-regular von Neumann subalgebra. If $z(\Sigma)$ equals the central support of $p$, then the same is true for $z(g_1 \Sigma g_1^{-1} \cap \cdots \cap g_n \Sigma g_n^{-1})$ and in particular
\begin{equation}\label{eq.conclusion}
B \embed{M} A \rtimes (g_1 \Sigma g_1^{-1} \cap \cdots \cap g_n \Sigma g_n^{-1})
\end{equation}
for all $g_1,\ldots,g_n \in \Gamma$
\end{corollary}

\section{Application to the proofs of Theorem \ref{thm.unique-Cartan}, \cite[Theorems 5.2 and 1.4]{PV09} and \cite[Theorem 4.1]{FV10}}

In the setup of Theorem \ref{thm.unique-Cartan} and \cite[Theorems 5.2 and 1.4]{PV09} we know that $\Gamma = \Gamma_1 *_\Sigma \Gamma_2$ is an amalgamated free product and we know that $A \rtimes \Gamma$ is a factor. We prove that in such a situation, whenever $B \subset p(A \rtimes \Gamma)p$ is a quasi-regular subalgebra, $z(\Sigma)$ can only take the values $0$ or $1$. A similar statement is true when $\Gamma = \HNN(H,\Sigma,\theta)$ is an HNN extension\footnote{The HNN extension $\HNN(H,\Sigma,\theta)$ is generated by a copy of $H$ and an element $t$, called stable letter, subject to the relation $t \sigma t^{-1} = \theta(\sigma)$ for all $\sigma \in \Sigma$.} of a countable group $H$ with a subgroup $\Sigma < H$ and an injective group homomorphism $\theta: \Sigma \recht H$.

The precise formulation goes as follows.

\begin{proposition} \label{prop.final}
Let $\Gamma$ either be an amalgamated free product $\Gamma  = \Gamma_1 *_\Sigma \Gamma_2$ with $\Gamma_1 \neq \Sigma \neq \Gamma_2$ or an arbitrary HNN extension $\Gamma = \HNN(H,\Sigma,\theta)$. Let $(A,\tau)$ be a tracial von Neumann algebra and $\Gamma \actson A$ a trace preserving action. Put $M = A \rtimes \Gamma$ and let $B \subset pMp$ be a quasi-regular von Neumann subalgebra. As above we define for every subgroup $\Lambda < \Gamma$, the projection $z(\Lambda) \in M \cap (A \rtimes \Lambda)'$ such that $p\rL^2(M) z(\Lambda)$ is the closed linear span of all $B$-$(A \rtimes \Lambda)$-subbimodules of $p \rL^2(M)$ that are finitely generated as a right Hilbert module.

The projection $z(\Sigma)$ belongs to the center of $M$.
So, if $M$ is a factor and if $B \embed{M} A \rtimes \Sigma$, then $z(\Sigma) = 1$ and
$$B \embed{M} A \rtimes (g_1 \Sigma g_1^{-1} \cap \cdots \cap g_n \Sigma g_n^{-1})$$
for all $g_1,\ldots,g_n \in \Gamma$.
\end{proposition}
\begin{proof}
Assume first that $\Gamma = \Gamma_1 *_\Sigma \Gamma_2$ is a non-trivial amalgamated free product.
We use Proposition \ref{prop.intersection} to prove that $z(\Sigma) = z(\Gamma_1)$. Once this is proven, by symmetry also $z(\Sigma) = z(\Gamma_2)$. But then $z(\Sigma)$ commutes with both $A \rtimes \Gamma_1$ and $A \rtimes \Gamma_2$, so that $z(\Sigma)$ belongs to the center of $M$.

Put $z_1 = z(\Gamma_1)$. Define $S \subset \Gamma$ as the set of elements $g \in \Gamma$ that admit a reduced expression starting with a letter from $\Gamma_2-\Sigma$. Whenever $g \in S$, we have $g \Gamma_1 g^{-1} \cap \Gamma_1 \subset \Sigma$. It follows from Proposition \ref{prop.intersection} that the projections $z_1 = z(\Gamma_1)$ and $u_g z_1 u_g^* = z(g \Gamma_1 g^{-1})$ commute and that
\begin{equation}\label{eq.intermediate}
z(\Sigma) \geq z(g \Gamma_1 g^{-1} \cap \Gamma_1) = u_g z_1 u_g^* \; z_1 \quad\text{for all}\;\; g \in S \; .
\end{equation}
We claim that
$$z_1 = \bigvee_{g \in S} u_g z_1 u_g^* \; z_1 \;\; .$$
To prove this claim, put
$$q = z_1 - \bigvee_{g \in S} u_g z_1 u_g^* \; z_1 \;\; .$$
Whenever $g \in S$, we have
$$q \; u_g q u_g^* = q z_1 \; u_g z_1 q u_g^* = q \; z_1 \; u_g z_1 u_g^* \; u_g q u_g^* = q \; u_g z_1 u_g^* \; z_1 \; u_g q u_g^* = 0 \; .$$
Take $a \in \Gamma_1 - \Sigma$ and $b \in \Gamma_2 - \Sigma$ and put $u_n = u_{(ba)^n}$. It follows that the projections $u_n q u_n^*$, $n \in \N$, are mutually orthogonal. Indeed, if $n < m$, we have
$$u_n q u_n^* \; u_m q u_m^* = u_n \bigl(q \; u_{m-n} q u_{m-n}^* \bigr) u_n^* = 0$$
because $(ba)^{m-n} \in S$. Since $\tau$ is a finite trace on $M$, it follows that $\tau(q) = 0$ and hence, $q=0$. This proves the claim. In combination with \eqref{eq.intermediate} it follows that $z(\Sigma) \geq z_1$. Hence, $z(\Sigma) = z_1$.

Next assume that $\Gamma = \HNN(H,\Sigma,\theta)$. Denote by $t \in \Gamma$ the stable letter. For every $n \geq 1$, one has $t^{-n} H t^n \cap H \subset \Sigma$. The same argument as in the case of amalgamated free products shows that $z(\Sigma) = z(H)$. We also have $\Sigma \subset t^{-1} H t$. Hence $z(\Sigma) \leq z(t^{-1} H t)$. Since $z(H) = z(\Sigma)$ and $z(t^{-1} H t) = u_t^* z(H) u_t$, we conclude that $z(H) \leq u_t^* z(H) u_t$. The left and right hand side have the same trace and hence must be equal. It follows that $z(\Sigma)$ commutes with $u_t$. Since $z(\Sigma)$ already commutes with $A \rtimes H$, it follows that $z(\Sigma)$ belongs to the center of $M$.

Finally consider the special case where $M$ is a factor and $B \embed{M} A \rtimes \Sigma$. The latter precisely means that $z(\Sigma) \neq 0$. Since $z(\Sigma)$ is a projection in the center of $M$, it follows that $z(\Sigma) = 1$. Then also $z(g \Sigma g^{-1}) = 1$ for all $g \in \Gamma$. By Proposition \ref{prop.intersection}, we get that
$$z(g_1 \Sigma g_1^{-1} \cap \cdots \cap g_n \Sigma g_n^{-1}) = 1$$
for all $g_1,\ldots,g_n \in \Gamma$. In particular,
$$B \embed{M} A \rtimes (g_1 \Sigma g_1^{-1} \cap \cdots \cap g_n \Sigma g_n^{-1})$$
for all $g_1,\ldots,g_n \in \Gamma$.
\end{proof}

\vspace{2ex}

{\small\parbox[t]{140pt}{Cyril Houdayer \\ CNRS-ENS Lyon\\ UMPA UMR 5669\\ 69364 Lyon cedex 7\\ France\\
{\footnotesize cyril.houdayer@ens-lyon.fr}}\hspace{15pt}\parbox[t]{130pt}{Sorin Popa\\ Mathematics Department\\ UCLA\\ CA 90095-1555\\ United States\\
    {\footnotesize popa@math.ucla.edu}}\hspace{15pt}\parbox[t]{150pt}{Stefaan Vaes\\ Department of Mathematics\\
    K.U.Leuven, Celestijnenlaan 200B\\ B--3001 Leuven\\ Belgium
    \\ {\footnotesize stefaan.vaes@wis.kuleuven.be}}}

\end{document}